\RequirePackage{fix-cm}
\documentclass[smallextended]{svjour3}       % onecolumn (second format)
\smartqed  % flush right qed marks, e.g. at end of proof
\usepackage{graphicx}
\usepackage{amsmath, amsfonts, amscd, amssymb, graphicx, color}%amsthm, 

\DeclareMathOperator{\Dom}{dom} 
\DeclareMathOperator{\barcone}{bar} 
\DeclareMathOperator{\Sol}{Sol} 
\DeclareMathOperator{\ri}{ri}
\DeclareMathOperator{\co}{co} 
\DeclareMathOperator{\Int}{int} 

\DeclareMathOperator{\aff}{aff}
\DeclareMathOperator{\dst}{d}

\DeclareMathOperator{\rec}{rec}
\DeclareMathOperator{\barc}{bar}

\def\bbar{\overline}
\def\R{\mathbb{R}}
\def\N{\mathbb{N}}
\def\<{\langle}
\def\>{\rangle}
\def\ds{\displaystyle}

\def\longto{\longrightarrow}

\begin{document}

\title{Separation of Convex Sets via Barrier Cones
\thanks{Dedicated to Professor Hoang Tuy.}
}
%\subtitle{Do you have a subtitle?\\ If so, write it here}

%\titlerunning{Short form of title}        % if too long for running head

\author{Huynh The Phung}

%\authorrunning{Short form of author list} % if too long for running head

\institute{Huynh The Phung \at
             Department of Mathematics, Hue College of Sciences, Hue University, Vietnam\\
              Tel.: +84-234-3822407\\
              %Fax: +123-45-678910\\
              \email{huynhthephung@gmail.com}   
}

\date{Received: 12 June 2018 / Accepted: 23 April 2020}
% The correct dates will be entered by the editor

\maketitle

\begin{abstract}
A closed convex subset of a normed linear space is said to have the strong separation property if it can be strongly separated from every other  disjoint closed and convex set  by a closed hyperplane. In this paper we  give some results on the separation of convex sets with noticing the role of barrier cones, develop some characterizations of  subsets having the  strong separation property, and apply them to consider a class of convex optimization problems.
\keywords{Convex set \and separation theorem \and barrier cone \and recession cone \and set having the strong separation property.}
% \PACS{PACS code1 \and PACS code2 \and more}
\subclass{46A55 \and 46B20 \and 52A05}
\end{abstract}

\section{Introduction}
\label{intro}

Let $C$ and $D$ be convex subsets of a  real normed linear space $X$ with dual space $X^*$. 
If there exists $x^* \in X^*\setminus \{0\}$ such that
$$\sup\{ \<x^*,c\>\mid c\in C\}  \le \inf\{\<x^*,d\> \mid d\in D\},$$ 
then we say that  $C$ and $D$ are separated.   
Furthermore, if   
$$\sup\{ \<x^*,c\>\mid c\in C\}< \inf\{\<x^*,d\> \mid d\in D\},$$
then $C$ and $D$ are said to be strongly separated.

A convex subset of\, $X$ is said to have the (strong) separation property if it can be (strongly) separated from every other  disjoint closed convex subset.

Let $C$ be a closed convex subset of $X$. We denote by $\rec(C)$ and $\barc(C)$, respectively,  the recession cone and the barrier cone of $C$, i.e.,
$$\rec(C):=\{v\in X\mid c+v\in C,\, \forall\, c\in C\};$$
$$\barc(C):=\{x^*\in X^*\mid \sigma_C(x^*)<+\infty\},$$
where $\sigma_C: X^*\to \bbar \R$ is the support function of $C$, defined by
$$\sigma_C(x^*)=\sup\{\<x^*, c\>: \, c\in C\},\; x^*\in X^*.$$ 

The set $C$ is called \emph{linearly bounded} if $\rec(C)=\{0\}$. It is obvious that a bounded subset is also linearly bounded. The set  
$C$ is said to be \emph{locally compact} if there exist  $c_0\in C$ and $r>0$ such that
\begin{equation}\label{eq2} 
 \bbar{B(c_0; r)}\cap C \text{ is compact},\end{equation}
where $\bbar{B(c_0; r)}$ denotes the closed ball of radius $r$ around $c_0$. 
It should be noted that, since $C$ is convex and closed, this definition does not depend on both $c_0$ and $r$, i.e., if \eqref{eq2} holds, then for every $c\in C$ and $s>0$,
$\bbar{B(c; s)}\cap C$  is also compact. 

The following results are well known (see, for instance, \cite{JD66,EkeTe,KN,K,R,Taylor}) in convex analysis.
\begin{theorem}\label{thm1} Let $C$ and $D$ be disjoint convex subsets of $X$. Then they are separated if at least one of the following conditions holds.

$(a)$ $\Int(C)\cup \Int(D)\neq \emptyset;$

$(b)$ $\dim(X)<\infty.$
\end{theorem}
\begin{theorem}\label{thm4}  Let $C$ and $D$ be the convex subsets of $X$. The following statements are equivalent.

$(a)$ $C$ and $D$ are strongly separated;

$(b)$ $\dst(C;D):=\inf\{\|c-d\|\mid c\in C,\, d\in D\}>0.$
\end{theorem}
\begin{theorem}\label{thm3} Let $C$ and $D$ be disjoint convex subsets in $X$. If 
\begin{equation}\label{eq3}
C \text{ or } D \text{ is weakly compact,}
\end{equation}
and the other is closed, then they are strongly separated.
\end{theorem}
\begin{corollary}\label{cor4} Let $C$ and $D$ be disjoint closed convex subsets of a reflexive Banach space $X$.
If  one of the sets is bounded, then they are strongly separated.
\end{corollary}

\begin{theorem}\label{thm2} Let $C$ and $D$ be disjoint closed convex subsets satisfying \begin{equation}\label{eq5}
\rec(C)\cap \rec(D)=\{0\}.
\end{equation}
If, in addition, $C$ or $D$ is locally compact, then they are strongly separated.
\end{theorem}

It is evident that any closed set in a finite-dimensional space is locally compact. Thus, a locally compact set may still be unbounded, and hence, may not be weakly compact. 

Since a convex set is linearly bounded whenever it is bounded, \eqref{eq5} is much weaker than \eqref{eq3}. Therefore, to compensate for that weakness,  in Theorem \ref{thm2} one of the sets is  required to be locally compact for a strong separation.

\begin{remark} For the strong separation, the condition \eqref{eq5} seems to be essential even in the case of finite-dimensional spaces. Indeed, it is obvious that the following subsets of $\R^2$
\begin{equation}\label{eq7}
C=\big\{(x,y)\in \R^2\mid x>0,\, y\ge \frac 1x\big\}\quad\text{and}\quad D=\{(x,0)\in \R^2\mid x\in \R\}\end{equation}
are convex, closed and disjoint, but are not strongly separated. The reason for this is that:
$$\rec(C)\cap\rec(D)=\{(u, 0)\mid u\ge 0\}\neq \{(0,0)\}.$$
\end{remark} 

\begin{remark}\label{ex1} In Corollary \ref{cor4}, if the underlying space is infinte-dimensional then the boundedness (or weak compactness) hypothesis of one of the subsets cannot be substituted by condition \eqref{eq5}. Indeed, consider the following subsets of the Hilbert space $l^2$:
$$C=\Big\{\xi=(x_n)\in l^2\,\Big|\, \sum_{n=1}^\infty \frac{x_n}n =1;\; x_n\ge 0,\, \forall n\Big\},$$
$$D=\Big\{\zeta=(y_n)\in l^2\,\Big|\, \sum_{n=1}^\infty \frac{y_n}{n+1} =1;\; y_n\ge 0,\, \forall n\Big\}.$$
Obviously, $C$ and $D$ are disjoint unbounded closed convex  subsets of $l^2$. Let $(\xi^k)\subset C$ and $(\zeta^k)\subset D$ be sequences defined by
$$\xi^k=(0,\ldots,0,k^{k-{\text{th}}},0,\ldots);\quad \zeta^k=(\frac 2{k+1},0,\ldots,0,k^{k-{\text{th}}},0,\ldots);\; k\in \N.$$
Since $\|\xi^k-\zeta^k\|_2=\frac 2{k+1}\to 0,$ $C$ and $D$  are not strongly separated. It should be noted that, although being  unbounded, both $C$ and $D$ are  linearly bounded, hence $\rec(C)\cap \rec(D)=\{0\}.$\end{remark}
\begin{remark}\label{rem1} The local compactness assumption on the sets in Theorem~\ref{thm2} seems a bit strong in the case of infinite-dimensional spaces. For example, consider the following subsets of $l^2$:
$$C=\Big\{x=(x_n)\in l^2\mid x_1\ge \Big(\sum_{i\neq 1}x_i^2\Big)^\frac{1}{2}\Big\},$$
$$D=\Big\{x=(x_n)\in l^2\mid x_2\ge 1+2\Big(\sum_{i\neq 2}x_i^2\Big)^\frac{1}{2}\Big\}.$$

Firstly, we have $rec(C)\cap \rec(D)=\{0\}$ because
$$\rec(C)= \Big\{v\in l^2\mid v_1\ge \Big(\sum_{i\neq 1}v_i^2\Big)^\frac{1}{2}\Big\},\; \rec(D)= \Big\{v\in l^2\mid v_2\ge 2\Big(\sum_{i\neq 2}v_i^2\Big)^\frac{1}{2}\Big\}.$$

It is easy to check that $C$ and $D$ are disjoint closed convex sets and are strongly separated by the vector $x^*=(1, -1, 0, 0, \ldots, )\in l^2$. 

On the other hand, by setting $e_1=(1, 0, 0, \ldots, )$, $e_2=(0, 1, 0, 0, \ldots)$ we have $e_1\in \Int C$ and $2e_2\in \Int D$. Thus, $C$ and $D$ are not locally compact, and hence, Theorem \ref{thm2} cannot be applied to establish a strong separation for them.
\end{remark}

Our first aim in this paper is \emph{to develop a new result on the strong separation of convex sets by imposing an assumption on the barrier cones of the sets in place of weak compactness or local compactness assumptions}.

From Theorem~\ref{thm1}, Theorem~\ref{thm3} and Theorem~\ref{thm2}, it follows that if $C$ has a nonempty interior or $X$ is finite-dimensional then $C$ has the separation property, and if $C$ is weakly compact or it is locally compact and linearly bounded then it has the strong separation property. Some further features of subsets having (strong) separation property have been established in the literature (for instance, see \cite{ErThe,GW}). Especially, in the case of Hilbert spaces, we have the  interesting result below. For a convex set $C\subset X$, let $\ri C$ denote its relative interior; that is,
$$\ri C:=\{x\in C\mid \exists\, \epsilon>0, B(x;\epsilon)\cap C\subset \aff(C)\},$$
where $\aff(C)$ is the affine hull of $C$ and $B(x; \epsilon)$ denotes the open ball with radius $\epsilon$ around $x$.

\begin{theorem}{\rm \cite[Theorem 2]{ErThe}}\label{thm10} 
An unbounded closed convex subset $C$ of a Hilbert space $X$ has the separation property if and only if $\aff(C)$ is a finite-codimensional closed affine subspace and $\ri C$ is nonempty.    
\end{theorem}  

Our second aim is \emph{to provide some necessary and/or sufficient conditions for a closed convex subset in a normed space to have the strong separation property.}

Recall that  if $M\subset \R^n$ is a nonempty closed convex set and  $f: \R^n \to \bbar R$ is a convex, lower semicontinuous and coercive function, then the optimization problem
$$\mathcal{P}(M; f):\, \begin{cases}
f(x) \to \inf,\\
x\in M
\end{cases}$$
has a nonempty compact solution set.    

The third aim is \emph{to prove a similar result for the convex programming problem with the constraint set having the strong separation property and the coerciveness assumption of the objective function is replaced by a weaker one.}

The rest of the paper is organized as follows: The next section will present a characterization for the interior of the barrier cones of convex sets in normed linear spaces. In Section 3 we develop a new result on strong separation with noticing the role of barrier cones. In Section 4 we provide some conditions for a closed convex set to have the strong separation property. Finally, Section 5 is devoted to considering convex optimization problems with constraint set having the strong separation property. 
\section{A characterization of the interior of the barrier cone}\label{sec:1}
In this section we try to characterize the interior of the barrier cone of a closed convex subset $C$ in a normed linear space $X$. We first note 
 that, since the support function $\sigma_C$ is sublinear and $\sigma_C(0)=0$, the barrier cone of $C$
is a convex cone containing the origin.

With the sets given in \eqref{eq7} we have
$$\barcone(C)=\{(u, v)\in \R^2\mid u\le 0,\; v\le 0\};\quad \barcone(D)=\{(0, v)\mid v\in \R\}.$$
Thus, $\Int\barcone(C)\neq \emptyset$ and $\Int\barcone(D)= \emptyset.$ 

It is well known that the weak$^*$-closure of $\barc(C)$ coincides with the polar cone of $\rec(C)$, i.e., 
$$\bbar{\barc(C)}^*=\rec(C)^0=\{x^*\in X^*\mid \<x^*, v\>\le 0,\, \forall\, v\in \rec(C)\}.$$
If $X$ is a reflexive Banach space then the norm-closure and the weak$^*$-closure of $\barc(C)$ coincide. Thus, we have
$$\bbar{\barc(C)}=\rec(C)^0.$$
However, this relation may fail in a general normed linear space. In \cite{AErThe} the authors have given a complete description of the norm-closure of $\barc(C)$ when  $C$ is a closed convex subset of a normed linear space $X$:
\begin{equation}\label{eq12}
\bbar{\barc(C)}=\left\{x^*\in X^*\,\Big|\, \lim_{r\to \infty}\left(\inf_{c\in C; \<x^*,c\>\ge r}\frac{\|c\|}{r}\right)=\infty\right\}.
\end{equation}
In fact, $\bbar{\barc(C)}$ can be represented in another form as stated below.
\begin{theorem}\label{thm11}
\begin{equation}\label{eq16}
\bbar{\barc(C)}=\left\{x^*\in X^*\,\Big|\, \limsup_{c\in C; \|c\|\to\infty}\frac{\<x^*, c\>}{\|c\|}\le 0\right\}.
\end{equation}
\end{theorem}  
\begin{proof} We need to show that, for every $x^*\in X^*$, 
\begin{equation}\label{eq25}
\lim_{r\to \infty}\left(\inf_{c\in C; \<x^*,c\>\ge r}\frac{\|c\|}{r}\right)=\infty \Leftrightarrow \limsup_{c\in C; \|c\|\to\infty}\frac{\<x^*, c\>}{\|c\|}\le 0.
\end{equation}
Since both sides of the relation hold for $x^*=0$ we may assume $x^*\neq 0$ and prove that the statements below are equivalent:
\begin{align*}
&(i) \quad  \forall\, M>0, \, \exists\, N>0,\, \forall\, r\ge N,\, \forall\, c\in C, \<x^*, c\>\ge r \Rightarrow \frac{\|c\|}r>M;\\
&(ii) \quad \forall\, \epsilon>0, \, \exists\, K>0,\, \forall\, c\in C, \|c\|\ge K \Rightarrow \frac{\<x^*, c\>}{\|c\|}<\epsilon.
\end{align*}
\noindent ($i \Rightarrow ii$). For every $\epsilon>0$ we set $M=\frac{1}{\epsilon}$. Then there exists $N>0$ satisfying ($i$). Let $K=\frac{N}{\epsilon}>0$. For every $c\in C$ such that $\|c\|\ge K$, by letting $r:=\<x^*, c\>$ we have:

$\bullet$ If  $\<x^*, c\>=r\ge N$ then $\frac{\|c\|}{\<x^*, c\>}=\frac{\|c\|}{r}>M$ and hence $\frac{\<x^*, c\>}{\|c\|}<\epsilon.$

$\bullet$ If $\<x^*, c\>< N$ then $\frac{\<x^*, c\>}{\|c\|}<\frac{N}{K}=\epsilon.$

\noindent ($ii \Rightarrow i$). For every $M>0$ we set $\epsilon=\frac 1M>0$ again. Then there exists $K>0$ satisfying ($ii$). Let now $N=K\|x^*\|$. For every $r\ge N$ and $c\in C$ such that $\<x^*, c\>\ge r$, we have
$$K\|x^*\|=N\le r\le \<x^*, c\>\le \|x^*\|\|c\|.$$
This shows that $\|c\|\ge K$, which, by ($ii$),  implies $\frac{r}{\|c\|}\le \frac{\<x^*, c\>}{\|c\|}<\epsilon$ and hence $\frac{\|c\|}{r}>M.$
\end{proof}

Inspired by this result we derive a characterization for the interior of $\barc(C)$ as below:
\begin{equation}\label{eq17}
\Int \barc(C)=\left\{x^*\in X^*\,\Big|\, \limsup_{c\in C; \|c\|\to\infty}\frac{\<x^*, c\>}{\|c\|}< 0\right\}.
\end{equation}
We state this fact in the following result.
\begin{theorem}\label{thm8} Let $x^*\in \barcone(C)$.  The following statements are equivalent:

$(a)$ $x^* \in \Int \barcone(C)$;

$(b)$ There exists $\gamma>0$ such that
\begin{equation}\label{eq27}
\sup_{c\in C\setminus B(0;\gamma)}\<x^*, c\>< \sigma_C(x^*);
\end{equation}

$(c)$  There exist positive numbers $\alpha$, $R$ such that
\begin{equation}\label{eq29}
\<x^*, c\> \le -\alpha \|c\|,\; \forall\, c\in C\setminus B(0;R);
\end{equation} 

$(d)$ $\ds \limsup_{c\in C; \|c\|\to\infty}\frac{\<x^*, c\>}{\|c\|}< 0,$

\noindent where, $B(0; \gamma)$ and $B(0; R)$ denote, respectively, the open balls of radii $\gamma$ and $R$ around the origin.  
\end{theorem}
\begin{proof} Since the equivalence between $(c)$ and $(d)$ is rather obvious, we only need to prove $(a) \Rightarrow (b) \Rightarrow (c) \Rightarrow (a)$.

$(a) \Rightarrow (b)$. Suppose that \eqref{eq27} fails to hold for every $\gamma>0$, or equivalently,
\begin{equation}\label{eq28}
\sup_{c\in C\setminus B(0;\gamma)}\<x^*, c\>=\sigma_C(x^*), \; \forall\, \gamma>0.
\end{equation}
Then there is a sequence $(c_n)\subset C$ such that $\|c_n\|\to \infty$ and 
$$\lim_{n\to\infty} \<x^*,c_n\>=\sigma_C(x^*).$$
Since the sequence $(c_n)$ is unbounded, by virtue of Banach-Steinhaus theorem, there exists $u^* \in X^*$ such that 
$$\limsup_{n\to\infty}\<u^*,c_n\>=\infty.$$
It implies that
$$\sigma_C(x^*+\lambda u^*)\ge \limsup_{n\to \infty} \<x^*+\lambda u^*,c_n\>=\infty,\; \forall\, \lambda>0.$$
In other words, $x^*+\lambda u^*\not\in \barcone(C)$, for every $\lambda>0.$ Thus, $x^*\not\in \Int \barcone(C).$ 

$(b) \Rightarrow (c)$ By \eqref{eq27} there exists $c_0\in C\cap B(0;\gamma)$ such that, for some $\varepsilon>0$, 
$$\sup_{c\in C\setminus B(0;\gamma)}\<x^*,c\>< \<x^*,c_0\>-\varepsilon.$$
Choose $R$ large enough such that $R> \gamma$ and $\<x^*,c_0\> - \frac{\varepsilon}{4\gamma}R\le 0.$ We shall prove that \eqref{eq29} holds for such $R$ and $\alpha:=\frac{\varepsilon}{4\gamma}$. 

Take $c\in C\setminus B(0; R)$ arbitrarily. Since $\|c\|\ge R>\gamma>\|c_0\|$, there exists $\lambda\in (0,1)$ such that $\|u\|=\gamma$ with $u=\lambda c+(1-\lambda)c_0\in C.$ We have
$$\gamma=\|\lambda c+(1-\lambda)c_0\|\ge \lambda\|c\|-(1-\lambda)\|c_0\|,$$
which implies that
\begin{equation}\label{eq30}
\lambda\le \frac{\gamma+\|c_0\|}{\|c\|+\|c_0\|}\le \frac{2\gamma}{\|c\|}.
\end{equation}
Since $u \in C\setminus B(0;\gamma)$, we have
$$\<x^*,c_0\>-\varepsilon> \<x^*,u\>=\lambda\<x^*,c\>+(1-\lambda)\<x^*,c_0\>,$$
which together with \eqref{eq30} implies that
$$\varepsilon< \frac{2\gamma}{\|c\|}(\<x^*,c_0\>-\<x^*,c\>),$$
or, 
$$\<x^*,c\>\le  -\frac{\varepsilon}{2\gamma}\|c\|
+\<x^*,c_0\>.$$
Noting that $\|c\|\ge R$ and $-\frac{\varepsilon}{4\gamma}R+\<x^*,c_0\>\le 0$, we have
$$\<x^*,c\>\le -\frac{\varepsilon}{4\gamma}\|c\|-\frac{\varepsilon}{4\gamma}R
+\<x^*,c_0\>\le -\alpha\|c\|.$$

$(c) \Rightarrow (a)$  If  \eqref{eq29} fulfills, then for every $u^*\in B(x^*;\alpha)$, we have
$$\<u^*,c\>\le \<x^*,c\>+\|u^*-x^*\|\|c\|\le \<x^*,c\>+\alpha\|c\|\le 0,\; \forall\, c\in C\setminus B(0;R),$$
and hence, 
$$\sigma_C(u^*)\le \max\{0, \sigma_{C\cap B(0;R)}(u^*)\}\le R\|u^*\|<\infty.$$
Thus, $B(x^*;\alpha)\subset \barcone(C)$, from which $(a)$ follows.
\end{proof}
\begin{corollary} $C$ is bounded if and only if $\barcone(C)=X^*$.
\end{corollary}

\begin{proof} Since  $\barcone(C)$ is a cone, $\barcone(C)=X^*$ if and only if $0 \in \Int\barcone(C)$. On the other hand, it follows from Theorem \ref{thm8} that $0\in \Int\barcone(C)$ if and only if there exists $\gamma>0$ such that $C\setminus B(0;\gamma)=\emptyset$, or equivalently, $C$ is bounded. 
\end{proof}
\section{Separation theorems via recession cone and barrier cone}\label{sec:2}
As we have seen in Theorem \ref{thm2}, for the strong separation of  unbounded subsets, besides condition \eqref{eq5}, the assumption of local compactness is also required. In the following discussion, instead of using local compactness assumption on the sets, we require one of their barrier cones to have a nonempty interior. The main result of the section is stated below.
\begin{theorem}\label{thm6} Let $C$ and $D$ be disjoint closed convex subsets of a reflexive Banach space, satisfying \eqref{eq5}. If, in addition, 
\begin{equation}\label{eq4}
(\Int \barcone(C))\cup (\Int \barcone(D))\neq \emptyset,
\end{equation} then $C$ and $D$ are strongly separated.
\end{theorem}
Before proceeding to the proof we prove the following lemmas.
\begin{lemma}\label{lem2} Let $C$ be a closed convex subset of $X$ and $(c_n)$ be a sequence in $C$ such that $\|c_n\|\to \infty$ and $$\frac{c_n}{\|c_n\|}\overset w {\longrightarrow} u\in X.$$
Then $u\in \rec(C)$.
\end{lemma}
\begin{proof} Take $c\in C$ we prove that $c+u\in C$. Since $\|c_n\|\to \infty$, 
$$v_n:=\Big(1-\frac{1}{\|c_n\|}\Big)c+ \frac{1}{\|c_n\|}c_n \in C,$$
for $n$ large enough (such that $1<\|c_n\|$). On the other hand, $(v_n)$ weakly converges to $c+u$. By noting that a closed convex set is also weakly closed, we deduce $c+u\in C.$ Since this inclusion holds for every $c\in C$,  it follows that $u\in \rec(C)$.   
\end{proof}
\begin{lemma}\label{lem3} Let $(c_n)$ and $(d_n)$ be sequences in $X$ such that
$\|c_n\|\to \infty$, and for some $r>0$, $\|c_n-d_n\|\le r$ for every $n$. If 
$$\frac{c_n}{\|c_n\|} \longto u, \text{ or } \frac{c_n}{\|c_n\|} \overset{w}\longto u,$$
with $u\in X$, then 
$$\frac{d_n}{\|d_n\|} \longto u, \text{ or } \frac{d_n}{\|d_n\|} \overset{w}\longto u, \text{ respectively.}$$
\end{lemma}
\begin{proof} Since
$$\Big\|\frac{c_n}{\|c_n\|}-\frac{d_n}{\|d_n\|}\Big\|\le \frac{\|c_n-d_n\|}{\|c_n\|}+\Big|\frac{1}{\|c_n\|}-\frac{1}{\|d_n\|}\Big|\|d_n\|\le \frac{2r}{\|c_n\|}\to 0,$$
we have
$$\frac{c_n}{\|c_n\|}-\frac{d_n}{\|d_n\|}\longto 0,$$
from which the lemma follows.
\end{proof}
\begin{proof}[of Theorem \ref{thm6}] Assume $\Int \barcone(C)$ is nonempty. We prove $\dst(C;D)>0$ by contradiction. Suppose that there exist sequences $(c_n)\subset C$, $(d_n)\subset D$ such that $\|c_n-d_n\|\to 0$. There are two cases depending on whether or not $\|c_n\|$ tends to $\infty$.

$\bullet$ $\|c_n\|\to \infty$. Since the space is reflexive, without loss of generality, we may assume that
$$\frac{c_n}{\|c_n\|} \overset w\longrightarrow u\in X,$$
and hence, from Lemma \ref{lem3}, 
$$\frac{d_n}{\|d_n\|} \overset w\longto u.$$
Thus, by Lemma \ref{lem2}, $u\in \rec(C)\cap \rec(D)$. 

Choose $x^*\in \Int \barcone(C)$ such that $x^*\neq 0$. By Theorem \ref{thm8}, for some $\alpha>0$ we have
$$\<x^*,\frac{c_n}{\|c_n\|}\>\le -\alpha,$$
for $n$ large enough. By letting $n\to \infty$, we obtain
$$\<x^*, u\>\le -\alpha<0,$$
which implies $u\neq 0$, contradicting \eqref{eq5}.

$\bullet$ $\|c_n\|\not\to \infty$. By restricting to a subsequence if necessary, we may assume that $(c_n)$ weakly converges to $u\in X$. However, in this situation, $(d_n)$ also weakly converges to $u$. Since $C$ and $D$ are convex and closed, they are weakly closed. Thus, $u\in C\cap D$, contradicting to the assumption that $C$ and $D$ are disjoint.\end{proof}
\begin{example}
Let $C$ and $D$ be the sets given in Remark \ref{rem1}. For each $x\in C$, we have $x_1\ge 0$ and
$$\|x\|_2^2=\sum_{i=1}^\infty x_i^2\le 2x_1^2.$$
It implies that
$$\<-e_1, x\>=-x_1\le -\frac 1{\sqrt 2}\|x\|_2;\; \forall\, x\in C.$$
Therefore, by Theorem \ref{thm8}, $-e_1\in \Int\barc(C)$. Applying Theorem \ref{thm6}, we deduce that $C$ and $D$ are strongly separated.  While, as mentioned in Remark~\ref{rem1}, Theorem \ref{thm2} cannot be applied to establish a strong separation here.  
\end{example}
\begin{remark} The sets $C$ and $D$ given in Remark ~\ref{ex1}\, satisfy the condition in \eqref{eq5}, but are not strongly separated. It is not difficult to verify that
$$\barcone(C)=\barcone(D)=\{(y_n)\in l^2\mid \sup_{n\ge 1} (ny_n) <\infty\},$$
and hence, $$\Int\barcone(C)= \Int\barcone(D)=\emptyset.$$
This fact shows that the condition \eqref{eq4} is crucial even in the case where $X$ is  a Hilbert space.
\end{remark}
\begin{example}\label{ex2} In Theorem \ref{thm6}, the assumption about reflexivity of the space is essential. Consider two subsets of  the nonreflexive space $l^1$:
$$C=\Big\{\xi=(x_n)\in l^1\,\Big|\, \sum_{n=1}^\infty x_n=1;\; x_n\ge 0,\, \forall n\Big\},$$
$$D=\Big\{\zeta=(y_n)\in l^1\,\Big|\, \sum_{n=1}^\infty \frac{ny_n}{n+1} =1;\; y_n\ge 0,\, \forall n\Big\}.$$
Obviously, $C$ and $D$ are disjoint bounded closed convex  subsets of $l^1$. Thus, condition \eqref{eq5} is fulfilled. In addition, since $C$ is bounded, $\Int\barcone(C)=l^\infty$. However, by letting  $(\xi^k)\subset C$ and $(\zeta^k)\subset D$ be the sequences defined by
$$\xi^k=(0,\ldots,0,1^{k-{\text{th}}},0,\ldots);\quad \zeta^k=\frac{k+1}{k}\xi^k;\; k\in \N,$$
we have $\|\xi^k-\zeta^k\|_1=\frac 1k\to 0.$ Thus, $C$ and $D$  are not strongly separated. 
\end{example}

As we have seen,  in a finite-dimensional space, any pair of disjoint closed convex sets satisfying \eqref{eq5} are strongly separated. In the case of infinite-dimensional spaces, besides the assumption of local compactness or condition \eqref{eq4}, condition in \eqref{eq5}  is also required for the strong separation of convex sets. 

Thus, condition \eqref{eq5} plays an important role in the strong separation. However, it should be noted that, this condition alone is not enough to yield even the (weak) separation of two disjoint closed convex subsets.  The following example illustrates this point.
\begin{example}
Let $X$ be a real Hilbert space, in which there exist two closed subspaces $M$ and  $N$ such that 
$M\cap N=\{0\}$, $M+N$ is dense but not closed in $X$, i.e., $M+N\neq X$ (see, \cite[Problem 2, p. 129]{JD}).

Take  $x_0\in X\setminus (M+N)$ and let $C=x_0-M$,  $D=N$. Thus,  $C$ and  $D$ are disjoint closed convex subsets. Furthermore, since $\rec(C)=M$ and  $\rec(D)=N$, $\rec(C)\cap \rec(D)=\{0\}$. We shall show that $C$ and $D$ are not separated. Suppose the contrary. Take $v\in X\setminus\{0\}$ such that
$$\<v,x_0-m\>\le \<v,n\>,\; \forall\, m\in M,\, n\in N.$$
It implies that
$$\<v,x_0\>\le \<v,x\>,\; \forall\, x\in M+N.$$
Since $M+N$ is dense in $X$,  it follows that $v=0$, a contradiction. Hence, $C$ and $D$ are not separated.
\end{example}

\section{Subsets having the strong separation property}\label{sec:3}
In this section we are interested in properties of  subsets having the strong separation property. 
As usual, let $S$ and $S^*$ denote the unit spheres in $X$ and $X^*$, respectively. Let $C$ be a closed convex subset of $X$. In some cases, the following conditions are needed:

$(A)$ $X$ is reflexive and $\Int\barcone(C)\neq\emptyset$.

$(B)$ $C$ is locally compact.

$(C)$ For some $r>0$ and  finite-dimensional subspace $Z$ we have:
$$C\subset B(0;r)+Z.$$

\begin{remark} The conditions $(A)$, $(B)$, $(C)$ are strongly independent in the sense that, each of them cannot be followed from the two remaining ones. This fact will be shown by the examples below.

$\bullet$ $C=\{(x,y)\in \R^2\mid y\le 0\}$ satisfies $(B)$ and $(C)$ but fails $(A)$.

$\bullet$ $C=\{x=(x_n)\in l^2\mid \|x\|_2\le 1\}$ satisfies  $(A)$ and $(C)$ but fails $(B)$.

$\bullet$ Let 
$$C=\{x:=(x_n)\in l^2\mid 0\le x_{n+1}\le \frac{n}{n+1}x_n,\, \forall\, n\ge 1\}.$$ 
Then $C$ is a closed convex cone. For every $x\in C$ we have
$$0\le x_1,\; 0\le x_2\le \frac{x_1}{2},\; 0\le x_3\le \frac{2x_2}{3}\le \frac{x_1}{3},\ldots , 0\le x_n\le \frac{x_1}{n},\ldots $$
It implies that
$$0\le x_1\le \|x\|_2=\sqrt{\sum_{i=1}^\infty x^2_i}\le x_1\sqrt{\sum_{i=1}^\infty \frac 1{i^2}}=\frac{\pi x_1}{\sqrt 6}.$$ Taking $x^*_0=(-1, 0, 0, \ldots)\in l^2$ we obtain
$$\<x^*_0, x\>=-x_1\le -\frac{\sqrt 6}{\pi} \|x\|_2,\; \forall\, x\in C,$$
which, by Theorem \ref{thm8}, implies  $x^*_0\in \Int \barcone(C).$ Thus, $C$ satisfies $(A)$.

We have 
\begin{align*}
C\cap \bbar{B(0;1)}&=\{x\in l^2\mid  \|x\|_2\le 1; 0\le x_{n+1}\le \frac{n}{n+1}x_n,\, \forall\, n\ge 1\}\\
&\subset E:= \{x\in l^2\mid 0\le x_n\le \frac 1n,\; \forall n\ge 1\}.
\end{align*}
Since $E$ is compact, $C$ satisfies $(B)$. Finally we show that $C$ does not satisfy $(C)$. Indeed, if $(C)$ holds then, since $C$ is a closed convex cone,  $C=\rec(C)\subset Z.$ But this is impossible because $Z$ is finite-dimensional while $C$ contains the following  infinite set of linearly independent vectors:
$$V=\{(1,0,0,\ldots), (1,\frac 12,0,\ldots), (1,\frac 12, \frac 13,0,\ldots),\ldots\}.$$    
\end{remark}

Theorem \ref{thm5} below will provide a necessary condition for a closed convex subset of $X$ to have the strong separation property.
\begin{lemma}\label{lem5} Let $C$ be a closed convex subset of $X$ and $(c_n)\subset C$ is a sequence such that $\|c_n\|\to +\infty$. If one of the conditions $(A)$, $(B)$ or $(C)$ is satisfied, then there exists  a subsequence $(c_{n_k})$ of $(c_n)$ such that, for some $x^*_0\in S^*$ and $\rho>0$, we have 
\begin{equation}\label{eq36}
\lim_{n_k\to \infty}\<x^*_0,\frac{c_{n_k}}{\|c_{n_k}\|}\>=-\rho.
\end{equation} 
\end{lemma}
\begin{proof}
We prove the lemma under each of the conditions: $(A)$, $(B)$ or $(C)$.

$(A)$ Take $x^*_0 \in S^*\cap \Int\barcone(C)$. By Theorem \ref{thm8}, for some $\alpha>0$ and $R>0$ we have
\begin{equation}\label{eq6}
\<x_0^*,c\>\le -\alpha\|c\|,\; \forall\, c\in C\setminus B(0;R).
\end{equation}
Since $X$ is reflexive, there exists a subsequence $(c_{n_k})$ of $(c_n)$ such that
$$\frac {c_{n_k}}{\|c_{n_k}\|}\overset{w}\longto s\in X.$$ 
This, together with \eqref{eq6}, implies \eqref{eq36} with $\rho=-\<x^*_0,s\>\ge \alpha>0.$ 

$(B)$ Since $C$ is locally compact, there exists a subsequence $(c_{n_k})$ of $(c_n)$ such that
$$\frac {c_{n_k}}{\|c_{n_k}\|}\longto s\in S.$$ 
By choosing $x^*_0\in S^*$ such that $\<x^*_0,s\>=-1$ we obtain \eqref{eq36} with $\rho=1$.
 
$(C)$  Since $C\subset B(0;r)+Z$, there exists a sequence $(z_n)\subset Z$ such that $\|z_n-c_n\|< r$ for all $n$, and hence, $\|z_n\|\to \infty.$ Since $\dim Z<\infty$, there exists a subsequence $(z_{n_k})$ of $(z_n)$ such that
$$\frac {z_{n_k}}{\|z_{n_k}\|}\longto s\in S.$$
From Lemma \ref{lem3} we also have
$$\frac {c_{n_k}}{\|c_{n_k}\|}\longto s\in S,$$
and by choosing $x^*_0$ as in the case of $(B)$ we obtain \eqref{eq36}.
\end{proof}

\begin{theorem}\label{thm5} Let $C\subset X$ be a closed convex subset having the strong separation property. In addition, at least one of the conditions $(A)$, $(B)$ or $(C)$ is satisfied. Then 
\begin{equation}\label{eq11}
\barcone(C)=\Int\barcone(C) \cup \{0\},
\end{equation}
that is to say, $x^*\in \Int\barcone(C)$ whenever $x^*\in \barcone(C)\setminus \{0\}.$
\end{theorem} 
\begin{proof} The proof is by contradiction. Suppose that there exists $x^*\in \barcone(C)\setminus \{0\}$ so that $x^*\not\in \Int\barcone(C)$. Since $\barc(C)$ is a cone we may assume $\|x^*\|=1$. By Theorem \ref{thm8}, \eqref{eq28} holds, and hence, there exists a sequence $(c_n)\subset C$ such that $\|c_n\|\to +\infty$ and
$$\<x^*,c_n\>\to \beta:=\sigma_C(x^*)<\infty.$$
Without loss of generality, we may assume that
$$\beta-\frac 1n<\<x^*,c_n\>\le \beta;\; \forall\, n.$$

From Lemma \ref{lem5}, without loss of generality we may assume that the sequence $\big(\frac{c_n}{\|c_n\|}\big)$ converges (strongly or weakly) to $s\in X$ and
$$\lim_{n\to \infty}\<x^*_0, \frac{c_n}{\|c_n\|}\>=\<x^*_0,s\>=-\rho<0,$$
for some $x^*_0\in S^*$ and $\rho>0$.

Choose  $v\in S$ such that $\<x^*,v\>>\frac 12$ and put  $d_n:=c_n+\frac 4nv$, for each integer $n\ge 1$. 
We now show that the following subset 
$$D=\bbar{\co}\{d_n \mid n\ge 1\}$$
is convex, closed and disjoint from $C$.

Clearly, $D$ is convex and closed. We prove $C\cap D=\emptyset$ by contradiction. 
Suppose that there exists $c_0\in C\cap D$. Since the sequence
$\big(\frac{c_n}{\|c_n\|}\big)$ converges (strongly or weakly) to $s\in X$, by virtue of Lemma \ref{lem3}, the sequence $\big(\frac{c_{n}-c_0}{\|c_{n}-c_0\|}\big)$ also converges to $s$.
Therefore, by setting 
$$s_n:=\frac{c_n-c_0}{\|c_n-c_0\|}, \quad t_n:=\|c_n-c_0\|$$ 
we have $s_n\in S$, $(s_n)$ converges (strongly or weakly) to $s$,  $t_n\to +\infty$,
$$c_n=c_0+t_ns_n;\; \forall\, n\ge 1,$$
and
$$\lim_{n\to\infty}\<x^*_0,s_n\>=\<x^*_0,s\>= -\rho<0.$$
Take $k\in \N$ large enough such that
\begin{equation}\label{eq1}
\<x^*_0,s_n\><-\frac \rho2,\; t_n>12;\; \forall\, n> k,
\end{equation}
and then set
\begin{equation}\label{eq19}
\gamma:=\max\{t_1,t_2,\ldots,t_k\}+1;\; \varepsilon:=\min\{\frac 1{2k\gamma}, \frac{2\rho}{5+k\gamma}\}< \frac 1{2k}.
\end{equation}
Since $c_0\in D=\bbar{\co}\{d_n \mid n\ge 1\}$, there exist nonnegative numbers $\lambda_1,\lambda_2,\ldots,\lambda_m$, with $m>k$, such that
$$\sum_{n=1}^m\lambda_n=1;\quad \Big\|\sum_{n=1}^m \lambda_n d_n-c_0\Big\|<\varepsilon.$$ 
Noting that $\|x^*\|=1$ we have
$$\varepsilon>\Big\|\sum_{n=1}^m \lambda_n d_n-c_0\Big\|\ge \<x^*,\sum_{n=1}^m \lambda_n d_n-c_0\> =\sum_{n=1}^m\lambda_n \<x^*,c_n+\frac 4nv\>-\<x^*,c_0\>$$
\begin{equation}\label{eq20}
\ge \sum_{n=1}^m \lambda_n (\beta-\frac 1n)+\sum_{n=1}^m \lambda_n \frac 2n-\beta  =\sum_{n=1}^m \frac{\lambda_n}{n}.
\end{equation}
It follows that
$$\varepsilon> \sum_{n=1}^k \frac{\lambda_n}{n}\ge \frac 1k\sum_{n=1}^k\lambda_n,$$
which, together with \eqref{eq19}, gives
\begin{equation}\label{eq21}
\sum_{n=1}^k\lambda_n<k\varepsilon\le \frac 12,
\end{equation}
and hence,
\begin{equation}\label{eq22}
\sum_{n=k+1}^m\lambda_n>\frac 12.
\end{equation}
On the other hand we also have
\begin{align*}
\varepsilon & >\Big\|\sum_{n=1}^m \lambda_n d_n-c_0\Big\|=\Big\|\sum_{n=1}^m \lambda_n (d_n-c_0)\Big\|=\Big\|\sum_{n=1}^m \lambda_n (t_ns_n+\frac 4n v)\Big\|\\
& \ge \Big\|\sum_{n=k+1}^m \lambda_n t_n s_n\Big\|-\Big\| \sum_{n=1}^k \lambda_n t_ns_n\Big\|-\Big\|\big(\sum_{n=1}^m \frac {4\lambda_n}n\big) v\Big\|.
\end{align*}
Noting that $v, s_n\in S$, $x^*_0\in S^*$ and $\<x^*_0,s_n\><-\frac \rho 2$ for $n>k$ one has
\begin{align*}
\varepsilon &> \<-x^*_0,\sum_{n=k+1}^m \lambda_n t_n s_n\> -\sum_{n=1}^k \lambda_n t_n-4\sum_{n=1}^m \frac {\lambda_n}n\\
&>\frac \rho 2\sum_{n=k+1}^m \lambda_nt_n-\sum_{n=1}^k \lambda_n t_n-4\sum_{n=1}^m \frac {\lambda_n}n.
\end{align*}
It follows that
\begin{equation}\label{eq31}
\frac \rho2 \sum_{n=k+1}^m \lambda_n t_n<\varepsilon+\sum_{n=1}^k \lambda_n t_n+4\sum_{n=1}^m \frac {\lambda_n}n.
\end{equation}
Since \eqref{eq1} and \eqref{eq22} we have
\begin{equation}\label{eq32}
\frac \rho2 \sum_{n=k+1}^m \lambda_n t_n>\frac \rho 2 \frac 12 12=3\rho.
\end{equation}
On the other hand, from \eqref{eq19} and \eqref{eq21} it follows that
\begin{equation}\label{eq33}
\sum_{n=1}^k \lambda_n t_n<\gamma\sum_{n=1}^k \lambda_n <k\gamma \varepsilon.
\end{equation}
Combining \eqref{eq31},\eqref{eq32}, \eqref{eq33}, \eqref{eq20} and the definition of $\varepsilon$ we obtain
$$3\rho<\varepsilon+k\gamma\varepsilon+4\varepsilon=(5+k\gamma)\varepsilon\le 2\rho$$
which is clearly absurd. Consequently, $C\cap D=\emptyset$.

Consequently, $D$ is a closed convex subset disjoint from $C$. On the other hand, since $\|c_n-d_n\|=\frac 4n\to 0$, $\dst(C;D)=0$, and hence, $C$ and $D$ are not strong separated. Thus, $C$ does not have the strong separation property. This completes the proof of the theorem.
\end{proof}

\begin{proposition}\label{prop1} If $C$ is unbounded and $\bbar{\aff(C)}\neq X$ then 
\begin{equation}\label{eq34}
\barcone(C)\neq \Int\barc(C)\cup \{0\}.
\end{equation}
\end{proposition}
\begin{proof}
Indeed, since $\bbar{\aff(C)}\neq X$  there exists $x^*\in X^*\setminus\{0\}$ such that
$$\<x^*,c\>=\alpha:=\sigma_C(x^*),\; \forall\, c\in C.$$
Hence, $x^*\in \barcone(C)$. On the other hand, since $C$ is unbounded, \eqref{eq28} holds. It now follows from Theorem \ref{thm8} that $x^*\not\in \Int\barc(C)$ and \eqref{eq34} is derived.
\end{proof}
From Theorem \ref{thm5} and Proposition \ref{prop1} we deduce the next corollary.
\begin{corollary}\label{cor1} Let $C$ be an unbounded closed convex subset of $X$ having the strong separation property. In addition, suppose that at least one of the conditions $(A)$, $(B)$, $(C)$ is satisfied. Then $\bbar{\aff(C)}=X$. Furthermore, if $\dim(C)<\infty$ then $\aff(C)=X$ and $\Int C\neq \emptyset.$
\end{corollary} 
As a converse of Theorem \ref{thm5} we have the following.
\begin{theorem}\label{thm7} Let $X$ be a reflexive Banach space and $C\subset X$ be a closed convex subset.  If $C$ has the separation property and \eqref{eq11} holds, then $C$ has the strong separation property.
\end{theorem} 
\begin{proof} We shall prove that, if  $C$ has the separation property but does not have the strong separation property, then \eqref{eq11} fails to hold.

Let $D$ be a closed convex subset of $X$, disjoint from $C$, but cannot be strongly separated from $C$. That is $\dst(C;D)=0$, i.e.,  there exist sequences $(c_n)\subset C$, $(d_n)\subset D$ such that $\|c_n-d_n\|\to 0$. If $\|c_n\|\not\to +\infty$ then, since $X$ is reflexive, by restricting to a subsequence if necessary, we may assume that $c_n\overset{w}\to \bar x$ and hence, $d_n\overset{w}\to \bar x$ too. Since $C$ and $D$ are (weakly) closed, $\bar x$ must belong to both of them, contradicting the assumption that they are disjoint. Consequently, 
\begin{equation}\label{eq23}
\|c_n\|\to +\infty.
\end{equation} 

On the other hand, by the separation property of $C$, there is a hyperplane $H(x^*;\alpha)$ ($x^*\neq 0$) separating $C$ and $D$, i.e.
\begin{equation}\label{eq24}
\<x^*,c\>\le \alpha\le \<x^*,d\>;\; \forall\, c\in C,\, \forall\,d\in D.
\end{equation}
It implies that $x^* \in \barc(C)$ and
$$\<x^*,c_n\>\le \alpha \le \<x^*,d_n\>;\; \forall\, n.$$
Noting that $\<x^*,d_n-c_n\>\le \|x^*\|\|d_n-c_n\|\to 0$ we derive the equalities:
$$\lim_{n\to \infty}\<x^*,c_n\>=\lim_{n\to \infty}\<x^*,d_n\>=\alpha,$$
which, together with  \eqref{eq23}-\eqref{eq24}, implies \eqref{eq28}. Thus, $x^*\in \barcone(C)\setminus \Int\barc(C)$. 
\end{proof}
\begin{theorem} Let $X$ be an infinite-dimensional real Hilbert space and $C$ be an unbounded closed convex subset of $X$. If, in addition, $C$ is locally compact, then it does not have the strong separation property. 
\end{theorem}
\begin{proof} Suppose the contrary that, $C$ has the strong separation property. By virtue of Theorem~\ref{thm10}, $\aff(C)$ is a finite-codimensional  closed affine subspace and $\ri C\neq \emptyset$.  
On the other hand, by Corollary \ref{cor1}, $\aff(C)=\bbar{\aff(C)}=X$, and hence, $\Int C=\ri C\neq \emptyset$. But this is impossible because $C$ is a locally compact subset in an infinite-dimensional space.
\end{proof}
\begin{theorem} Let $X$ be a real Hilbert space and $C\subset X$ be an unbounded closed convex subset satisfying either condition $(A)$ or $(C)$. Then $C$ has the strong separation property if and only if\, $\Int C$ is nonempty and \eqref{eq11} holds. 
\end{theorem}
\begin{proof}\

 If $C$ has the strong separation property then, by Theorem~\ref{thm10}, Theorem~\ref{thm5}, and Corollary \ref{cor1} we deduce that $\Int C$ is nonempty and \eqref{eq11} holds. 

Conversely, if\, $\Int C$ is nonempty and \eqref{eq11} holds then, by  Theorem~\ref{thm1} and Theorem~\ref{thm7}, $C$ has the strong separation property.
\end{proof}
\begin{corollary}\label{cor3} Let $C$ be a closed convex subset in a finite-dimensional space $X$. Then, $C$ has the strong  separation property if and only if \eqref{eq11} holds. Furthermore, if $C$ is unbounded and $C$ has the strong  separation property then $\Int C$ is nonempty.
  \end{corollary}
\begin{proof} If $X$ is finite-dimensional then it is reflexive and every closed  convex subset of $X$ is locally compact and has the separation property. The conclusion of the corollary therefore follows directly from Theorem~\ref{thm5} and Theorem \ref{thm7}.  
\end{proof}
\begin{remark}
Corollary~\ref{cor3} shows that, in finite-dimensional spaces, apart from bounded subsets, every unbounded closed convex subset  also has the strong separation property whenever the condition \eqref{eq11} is fulfilled. The example below presents a set of this type.
\end{remark}
\begin{example}\label{ex4} The following subset
$$C=\{(x,y)\in \R^2\mid y\ge x^2\}$$
is convex, closed and unbounded. It is not hard to verify that
$$\sigma_C(u,v)=\begin{cases}
+\infty,& \text{ if } (v>0) \text{ or } ((v=0) \text{ and } (u\neq 0)),\\
0,& \text{ if } u=v=0,\\
-\frac{u^2}{4v},& \text{ if } v<0. 
\end{cases}
$$
Consequently, $\barc(C)=\{(0,0)\}\cup\{(u,v)\mid v<0\}$, $\Int\barc(C)=\{(u,v)\mid v<0\}$. Thus, the condition \eqref{eq11} holds, and $C$ has the  strong separation property. 
\end{example}
\begin{example} Consider the subset of $\R^2$:
$$C=\{(x,y)\in \R^2\mid \exp(x)-y\le 0\}.$$ 
Since
$$\sigma_C(u,v)=\begin{cases}
+\infty,& (u<0 \text{ or } v\ge 0)\text{ and }((u,v)\neq (0,0)),\\
u\ln(-\frac{u}{v})-u,& u>0>v,\\
0,& v\le 0=u,
\end{cases}$$
$\barc(C)=\{(u,v)\mid u\ge 0>v\}\cup\{(0,0)\}$. Thus, 
$$\Int\barc(C)\cup\{(0,0)\}=\{(u,v)\mid u> 0>v\}\cup\{(0,0)\}\neq \barc(C).$$
It implies that $C$ does not have the strong separation property. 
\end{example}
\section{Application to a convex optimization problem}\label{sec:4}
In this section we shall establish some results for a convex optimization problem whose constraint set has the strong separation property. We assume throughout the section that $f:\R^n\to \bbar \R$ is a proper convex, lower semicontinuous function and $M\subset \R^n$ is a nonempty closed convex set. Consider the  optimization problem:
$$\mathcal{P}(M; f):\; \begin{cases}
f(x) \to \inf,\\
x\in M,
\end{cases}$$ in which we seek $\bar x\in M$ such that 
$$f(\bar x)=\bar f:=\inf\{f(x):\; x\in M\}.$$
The solution set of $\mathcal{P}(M; f)$ is denoted by $\Sol(M; f)$, that is,
$$\Sol(M; f)=\{\bar x\in M\mid f(\bar x)=\bar f\}.$$

The horizon function $f^\infty: \R^n\to \bbar{\R}$, associated with $f$ is defined by
$$f^\infty(v):=\lim_{\lambda\to +\infty}\frac{f(x_0+\lambda v)-f(x_0)}{\lambda},$$
with some $x_0\in \Dom f$. In fact, such a limit is independent of $x_0\in \Dom f$. The function $f^\infty$ is proper, sublinear and lower semicontinuous  (see for example \cite{RW}).  
$f$ is said to be coercive if  
$$\lim_{\|x\|\to \infty}f(x)=+\infty.$$
Since $f$ is convex on a finite-dimensional space, it is not difficult to verify that $f$ is coercive if and only if 
$$\liminf_{\|x\|\to \infty}\frac{f(x)}{\|x\|}>0,$$
or, equivalently,
 \begin{equation}\label{eq8}
\forall\, v\neq 0,\;  f^\infty(v)>0.
 \end{equation}
 
It is well known that, if the objective function $f$ is coercive and the constraint set $M$ is closed, then $\Sol(M; f)$ is nonempty and compact. In the following, we show that if $M$ has the strong separation property then, in order for the solution set to be compact, $f$ need not be coercive, instead, it is required to satisfy the next weaker condition:
\begin{equation}\label{eq10}
\forall\, 0\neq v\in C(f^\infty; 0), \exists\, \tilde x\in \Dom f, \ds\lim_{\lambda\to +\infty}f(\tilde x+\lambda v)=-\infty,
\end{equation}
where
$$C(f^\infty; 0):=\{v\in \R^n\mid f^\infty(v)\le 0\}.$$
This fact is stated in the following theorem. 
\begin{theorem}\label{thm9} If $M$ has the strong separation property, $f$ is bounded below on $M$ and satisfies condition \eqref{eq10}, then the solution set of $\mathcal{P}(M; f)$ is nonempty and compact.
\end{theorem}
\begin{proof} Since $f$ is convex and lower semicontinuous, $\Sol(M; f)$ is a closed convex set. Suppose that $\Sol(M; f)$ is not compact or empty. Then, there exists a sequence $(x_n)\subset M$ such that  $\|x_n\|\to +\infty$ and
$$\lim_{n\to \infty} f(x_n)=\bar f.$$
By an argument analogous to the proof of Lemma \ref{lem5} (under condition $(C)$), we may assume that
$$\frac{x_n}{\|x_n\|}\to s\in S.$$
Take $x_0\in M$. By Lemma \ref{lem2} and Lemma \ref{lem3}, we have $s\in \rec(M)$, 
$$s_n:=\frac{x_n-x_0}{\|x_n-x_0\|}\to s,$$
and $x_n=x_0+t_ns_n$ with $t_n=\|x_n-x_0\|\to +\infty$.

Fix a number $\lambda>0$. For $n$ large enough, one has $\lambda<t_n$ and
$$\frac{f(x_0+\lambda s_n)-f(x_0)}\lambda\le \frac{f(x_0+t_ns_n)-f(x_0)}{t_n}=\frac{f(x_n)-f(x_0)}{t_n}.$$
Since $f(x_n)\to \bar f$, the right-hand side of the inequality tends to $0$ while the left-hand side tends to $\frac{f(x_0+\lambda s)-f(x_0)}\lambda$, when $n\to +\infty$. Consequently, 
$$\frac{f(x_0+\lambda s)-f(x_0)}\lambda\le 0;\; \forall\, \lambda>0,$$
and hence, $f^\infty(s)\le 0$. Because $f$ satisfies condition \eqref{eq10}, there exists $\tilde x\in \Dom f$ such that
\begin{equation}\label{eq14}
\lim_{\lambda\to\infty}f(\tilde x+\lambda s)=-\infty.
\end{equation}

We shall show that the following straight line
$$L=\{\tilde x+\lambda s\mid \lambda\in \R\}$$
does not intersect $M$. Assume the contrary. Let $\lambda_0\in \R$ such that $\tilde x+\lambda_0s\in M$. Since $s\in \rec(M)$, $\tilde x+\lambda s\in M$ for all $\lambda \ge \lambda_0$. This together with \eqref{eq14}  implies that $\bar f=-\infty$, contradicting the fact that $f$ is bounded below on $M$.

Since $L$ is convex and disjoint from $M$, there exists a vector $x_0^*\in \R^n\setminus\{0\}$ separating $L$ and $M$; that is to say,
$$\sup\{\<x_0^*, x\>\mid x\in M\}\le \inf\{\<x_0^*, y\>\mid y\in L\}.$$
Thus, $x_0^*\in \barc(M) \setminus\{0\}$. Since $M$ has the strong separation property, it follows from Theorem \ref{thm5} and Theorem \ref{thm8} that 
$$\lim_{\overset{x\in M}{\|x\|\to\infty}}\<x_0^*, x\>=-\infty.$$ 
Since $x_0+\lambda s\in M$, for all $\lambda>0$, it implies that $\<x_0^*,s\><0$. On the other hand, because $\<x_0^*, \cdot\>$ is bounded below on $L$, we have $\<x_0^*,s\>=0$. This contradiction completes the proof.  
 \end{proof} 
\begin{example} Let consider the problem $\mathcal{P}(M; f)$ with
$$M=\{(x, y)\in \R^2\mid y\ge x^2\},$$
and
$$f(x, y)=y+x^2,\; (x, y)\in \R^2.$$
The set $M$ has the strong separation property as shown in Example \ref{ex4}. The function $f$ is bounded below on $M$ by $0$. On the other hand, 
\begin{align*}
f^\infty(u, v)&=\lim_{\lambda\to +\infty}\frac{f((0,0)+\lambda(u,v))-f(0,0)}{\lambda}\\
&=\lim_{\lambda\to +\infty}\frac{\lambda v+\lambda^2 u^2}{\lambda}=\begin{cases}
+\infty,& u\neq 0,\\
v,& u=0.
\end{cases}
\end{align*}
Therefore,
$$C(f^\infty; 0)=\{(0, v)\mid v\le 0\}.$$
For every $(0,0)\neq (u,v)\in C(f^\infty; 0)$, that is, $u=0$ and $v<0$, we have
$$\lim_{\lambda \to +\infty} f((0,0)+\lambda(0, v))=-\infty.$$
Thus, $f$ satisfies  the condition \eqref{eq10}. By virtue of Theorem \ref{thm9},  $\Sol(M; f)$ is nonempty and compact. In fact, by solving directly we can derive the solution set $\Sol(M; f)=\{(0,0)\}$.
It should be noticed that the function $f$ is not coercive since $f^\infty(0, v)<0$ for all $v<0$.    
\end{example} 
\begin{example}\label{ex3} In Theorem \ref{thm9}, if $f$ is not coercive then the assumption that $M$ has the strong separation property is essential and cannot be dropped. Let consider the optimization problem $\mathcal{P}(M; f)$  with
$$M=\{(x,0)\mid x\in \R\}\subset \R^2,$$
and 
$$f(x,y)=\begin{cases}
\exp(-x)-\sqrt{xy},& \text{if } x\ge 0 \text{ and } y\ge 0,\\
+\infty,& \text{if } x<0 \text{ or } y<0.
\end{cases}$$
We can verify that $f$ is proper, convex and lower semicontinuous on $\R^2$. Besides, 
$$f^\infty((u,v))=\lim_{\lambda \to +\infty}\frac{f(\lambda u,\lambda v)-f(0,0)}{\lambda}=\begin{cases}
-\sqrt{uv},& \text{if } u\ge 0 \text{ and } v\ge 0,\\
+\infty,& \text{if } u<0 \text{ or } v<0.
\end{cases}$$
So, if $(0,0)\neq (u,v)\in C(f^\infty; 0)$ then $u\ge 0$, $v\ge 0$, $u+v>0$, and hence, by taking $(\tilde x, \tilde y)=(1,1)$, we have
$$\lim_{\lambda\to +\infty}f((\tilde x, \tilde y)+\lambda(u,v))=\lim_{\lambda\to +\infty} [\exp(-1-\lambda u)-\sqrt{(1+\lambda u)(1+\lambda v)}]=-\infty.$$
That means condition \eqref{eq10} holds. Furthermore, $f$ is bounded below (by $0$) on $M$. However, it is easy to see that $\bbar f=0$ and $\Sol(M; f)=\emptyset$. This happens because $M$ does not have the strong separation property and $f$ is not coercive. 
\end{example}
Sometimes, the constraint set $M$ is defined by a system of convex inequalities as follows:
\begin{equation}\label{eq13}
M=\{x\in \R^n\mid f_i(x)\le 0,\; 1\le i\le m\},
\end{equation}
where $f_i$, $1\le i\le m$, are convex functions on $\R^n$. In order for the set given in \eqref{eq13} to have the strong separation property, each constraint function is required to satisfy the following condition: 
\begin{equation}\label{eq9}
\forall\, 0\neq v\in C(f_i^\infty; 0), \forall\, x\in \Dom f_i, \ds\lim_{\lambda\to +\infty}f_i(x+\lambda v)=-\infty.
\end{equation}
 \begin{theorem}
Assume that $f_i: \R^n \to \R$, $1\le i\le m$, are convex functions satisfying condition \eqref{eq9}. Then the set  $M$ defined as \eqref{eq13} has the strong separation property.   
\end{theorem}
\begin{proof}
Suppose the contrary. Let $D\subset \R^n$ be a closed convex subset, disjoint from $M$, but cannot be strongly separated from $M$. It follows from Theorem~\ref{thm2} that, there exists $0\neq v\in \rec(M)\cap \rec(D)$. Take $x_0\in M$ and $y_0\in D$. Since $x_0+\lambda v\in M$ for all $\lambda>0$, we have
$$f_i^\infty(v)=\lim_{\lambda\to +\infty}\frac{f_i(x_0+\lambda v)-f_i(x_0)}{\lambda}\le     \lim_{\lambda\to +\infty}\frac{-f_i(x_0)}{\lambda}=0;\; 1\le i\le m.$$
Because $f_i$ satisfies condition \eqref{eq9} we have
$$\lim_{\lambda\to +\infty}f_i(y_0+\lambda v)=-\infty; \; 1\le i\le m.$$
Consequently, there exists $\lambda>0$ such that $f_i(y_0+\lambda v) \le 0$, $1\le i\le m$, or
$y_0+\lambda v\in M$. On the other hand, since $v\in \rec(D)$, $y_0+\lambda v\in D$. Thus, $M\cap D\neq \emptyset$, contradicting the fact that $M$ and $D$ are disjoint.   
\end{proof}
\begin{corollary} Let $f_i:\R^n\to \R$, $1\le i\le m$, be convex functions satisfying condition \eqref{eq9}, $f_0: \R^n \to \R$ be a proper, convex and lower semicontinuous function satisfying condition \eqref{eq10}. Then the solution set of the following optimization problem
$$\mathcal{P}(f_1,f_2,\ldots,f_m; f_0):\begin{cases}
f_0(x)\to \inf,\\
x\in \R^n,\\
f_i(x)\le 0,\, 1\le i\le m
\end{cases}$$
is nonempty and compact.
\end{corollary}
\begin{remark}  From assumptions imposed on convex functions we observe that \eqref{eq8} $\Rightarrow$ \eqref{eq9} and \eqref{eq9} $\Rightarrow$ \eqref{eq10}. However, the converses are not true. For example, the function $f$ given in Example \ref{ex3} satisfies \eqref{eq10}, while, by taking $(1,0)\in C(f^\infty; 0)$ and $(x, y)=(0,0)\in \Dom f$ we have
$$\lim_{\lambda\to +\infty} f((0,0)+\lambda (1,0))=0>-\infty.$$
Thus, $f$ does not satisfy condition \eqref{eq9}.  Also, it is not hard to verify that the following function
$$f(x)=\begin{cases}
-\sqrt x,& \text{ if } x\ge 0,\\
+\infty, &\text{ if } x<0
\end{cases}$$
is proper, convex, lower semicontinuous on $\R$ satisfying condition \eqref{eq9}, but it is not coercive.
\end{remark}
\section{Conclusion} In this paper we have studied strong separation of convex sets and characterization of sets having the strong separation property by using results on the barrier cones of convex sets. We provide a full description of the interior of the barrier cone of a convex set, prove a new strong separation theorem under an assumption on the barrier cones  instead of local compactness or weak compactness assumptions on the sets. We also develop some necessary and/or sufficient conditions for a closed convex set to have the strong separation property. The non-emptiness and compactness of the solution set in a convex optimization problem whose constraint set has the strong separation property are also considered in the paper.

\begin{acknowledgements}
The author would like to thank the referees for their helpful comments and valuable suggestions.\end{acknowledgements}

\end{document}